\documentclass[letterpaper, 10 pt, conference]{ieeeconf}  

\IEEEoverridecommandlockouts                              
\usepackage{amsmath}
\usepackage{amssymb}
\usepackage{algorithm}
\usepackage{algpseudocode}
\usepackage{bbm}
\newtheorem{theorem}{Theorem}
\newtheorem{proposition}{Proposition}
\newtheorem{lemma}{Lemma}
\newtheorem{corollary}{Corollary}
\DeclareMathOperator*{\argmin}{arg\,min}
\usepackage{graphicx} 

\title{\LARGE \bf
Local Water Storage Control for the Developing World
}

\author{Yonatan Mintz, Zuo-Jun Max Shen, and Anil Aswani
\thanks{This work was supported in part by the Philippine-California Advanced Research Institutes (PCARI) and NSF Award CMMI-1450963.}
\thanks{Y. Mintz, Z.-J. Shen, and A. Aswani are with the Department of Industrial Engineering and Operations Research, University of California, Berkeley, CA 94720, USA 
        {\tt\small ymintz@berkeley.edu, maxshen@berkeley.edu, aaswani@berkeley.edu}}%
}

\begin{document}

\maketitle
\thispagestyle{empty}
\pagestyle{empty}

\begin{abstract}

Most cities in India do not have water distribution networks that provide water throughout the entire day. As a result, it is common for homes and apartment buildings to utilize water storage systems that are filled during a small window of time in the day when the water distribution network is active.  However, these water storage systems do not have disinfection capabilities, and so long durations of storage (i.e., as few as four days) of the same water leads to substantial increases in the amount of bacteria and viruses in that water.  This paper considers the stochastic control problem of deciding how much water to store each day in the system, as well as deciding when to completely empty the water system, in order to tradeoff: the financial costs of the water, the health costs implicit in long durations of storing the same water, the potential for a shortfall in the quantity of stored versus demanded water, and water wastage from emptying the system.  To solve this problem, we develop a new Binary Dynamic Search (BiDS) algorithm that is able to use binary search in one dimension to compute the value function of stochastic optimal control problems with controlled resets to a single state and with constraints on the maximum time span in between resets of the system.

\end{abstract}

\section{INTRODUCTION}
Safe drinking water is not widely available in much of the developing world. For instance, the majority of cities in India do not have continuous water supply within the existing distribution infrastructure, and on average citizens get water for a 3 to 4 hour period during the day \cite{bosch2006,um2007}. As a result, most households use communal and personal water storage containers to ensure adequate water supply throughout the day; these water storage containers are filled every 1 or 2 days during periods of water availability \cite{um2007,franceys2010}.  However, water stored in these containers often becomes contaminated \cite{zerah2000,evison2001,coelho2003controlling}:  Longer durations of storing the same water leads to increased amounts of bacterial and viral pathogens. 

In this paper, we formulate the problem of deciding how much water to store each day in a local water storage container as a stochastic optimal control problem.  The challenge is how to tradeoff the financial costs of purchasing the water, the degradation in water quality (and associated implicit health costs) from long storage durations, the potential for a shortfall in the quantity of stored versus demanded water, and water wastage.  The inherent uncertainty in daily water demand makes our choice of stochastic optimal control natural.  Furthermore, our goal is to construct a control policy that is easily implementable via a lookup table that could be distributed through paper pamphlets or internet websites.

In addition to studying the stochastic optimal control problem associated with local water storage container, our second contribution in this paper is to develop a new algorithm for solving stochastic optimal control problems with the specific structure of controlled resets to a single state and with constraints on the maximum time in between system resets.  Unlike value and policy iteration, which require finding a fixed point in an infinite dimensional functional space \cite{bertsekas1995}, we develop a Binary Dynamic Search (BiDS) algorithm that converts the control problem into finding a fixed point in a vector-space using binary search.  Though value and policy iteration converge exponentially, the convergence rate can be practically slow when the discount factor is close to 1 \cite{bertsekas1995,zidek2016stochastic}.  Our BiDS algorithm can compute the value functions of problems with controlled resets to a single state using substantially less computation because it uses binary search. 

\subsection{Energy Storage Control}

Though local water storage is nominally similar to control of (mainly electrical) energy storage \cite{bitar2011role,moura2011stochastic,taylor2013competitive,liu2013planning,tang2014dynamic,zhang2015stochastic,qin2016submodularity,li2016mean}, there are several differences in the underlying physics and domain settings that make the two problems very distinct.  The biggest difference is that stored energy dissipates (or decays) in quantity over time, whereas local water storage tanks are typically closed and hence have minimal water evaporation \cite{zerah2000,evison2001,coelho2003controlling}.  Another important contrast is that the quality of stored water decays \cite{zerah2000,evison2001,coelho2003controlling}, whereas the quality of stored energy is directly related to the stored quantity.  Because the quality of stored water decays, our formulation considers the possibility of flushing all water from the storage container; in contrast, energy storage does not have a corresponding control action to empty the storage.

\subsection{Inventory and Supply Chain Management}

Our problem of local water storage control is also similar to inventory and supply management problems \cite{wagner1958,scarf1959}.  More specifically, it is related to the setting of managing perishable inventory \cite{hsu2000,hsu2001,pierskalla1972,rosenfield1989,rosenfield1992,shen2011}, which has been mainly explored in the context of healthcare inventory management. In these perishable inventory models, it is assumed that expired units of inventory can be disposed of individually, which allows salvage and disposal costs to be incorporated into the inventory's holding cost; hence, policies of strategic inventory removal can be employed \cite{pierskalla1972,rosenfield1989,rosenfield1992,shen2011}. However, in our setting, since fresh water may mix with still water, individuals cannot dispose of single units of their stock strategically. Instead, individuals  must  make a decision of either keeping or disposing of their entire stock of water to reduce the risk of drinking contaminated water.

\subsection{Numerical Solution of Stochastic Optimal Control}

The workhorse for numerical solution of stochastic optimal control is value and policy iteration \cite{bertsekas1995}, which require finding a fixed point in an infinite dimensional space.  For high-dimensional state spaces, exact calculation for these approaches is numerically difficult (i.e., the so-called \emph{curse of dimensionality}).  And so a number of approximate dynamic programming approaches \cite{bertsekas1995,de2003linear,powell2007approximate,ryzhov2012knowledge,kariotoglou2013approximate,haskell2016empirical} have been developed, where the main idea is that approximate dynamic programming performs computation with an inexact but tractable representation of the value function or policy.  

However, the convergence rate of (exact or approximate) value and policy iteration is governed by the discount factor, and so convergence is slow for discount factors close to 1 \cite{bertsekas1995,zidek2016stochastic}.  This motivates development of our BiDS algorithm, which uses the special structure of stochastic control problems with controlled resets to a single state and with a constraint on the amount of time in between resets.  Our BiDS algorithm finds a fixed point in a vector-space using binary search.  Since the state space is small for our local water storage control problem, our numerical results solve the exact dynamic program; however, in principle our BiDS algorithm could be used for approximate dynamic programming by using inexact representations of the value function.

\subsection{Outline}
 
We first describe in Sect. \ref{sec:wat_stor_mode} the stochastic optimal control problem for managing local water storage.  We also describe a stochastic model for daily water demand.  Section \ref{sec:rcmba} generalizes this into a broader class of reset control problems (i.e., stochastic optimal control with controlled resets to a single state and with a constraint on the maximum timespan in between system resets).  Next we develop essential elements of the BiDS algorithm for solving reset control problems, and prove the correctness of the value function and control policy computed by BiDS.  We conclude with Sect. \ref{sect:nswsm}, which uses BiDS to numerically solve the local water storage control problem for a specific set of numeric problem parameters.

\section{Description of Water Storage Model}

\label{sec:wat_stor_mode}

In this section, we first describe the stochastic optimal control problem associated with managing a local water storage container.  A key element of this model is the stochastic distribution used to model daily water demand.  We discuss one common stochastic model for water demand in a household and provide an algorithm to generate random samples from this distribution, which is useful for numeric computations when solving the stochastic optimal control.

\subsection{Stochastic Optimal Control Formulation}

The subscript $n$ denotes the corresponding variable for the $n$-th day.  We define the following variables:
\begin{itemize}
\item $x_n$ - amount of water stored in the tank;
\item $t_n$ - number of days since the tank was last emptied;
\item $u_n$ - control of adding new water to the tank;
\item $d_n$ - amount of water consumed;
\item $f_n$ - control of emptying the entire tank;
\end{itemize}
We also have the following dynamics to describe the water storage tank:
\begin{align}
&\begin{aligned}
x_{n+1} &= \begin{cases} (x_n + u_n - d_n)^+, &\text{if } f_n = 0\\
(u_n - d_n)^+, &\text{if } f_n = 1\end{cases}\\
&=\big((1-f_n)\cdot x_n + u_n - d_n\big)^+
\end{aligned}
\label{eq:x_dynam}\\
&\begin{aligned}
t_{n+1} &= \begin{cases} t_n +1,& \text{if } f_n = 0\\
1,&\text{if } f_n = 1
\end{cases} \\
&= \big(t_n\cdot(1-f_n) + 1\big)
\end{aligned}
\end{align}
where $(x)^+ = \max\{x,0\}$ and $(x)^-=\min\{x,0\}$.  Let $q: \mathbb{R}_+\rightarrow\mathbb{R}_+$ be a strictly increasing nonnegative function, and suppose $c,p,c_f\in\mathbb{R}_+$ are fixed positive constants.  The goal is to solve the following stochastic control problem
\begin{equation}
\begin{aligned}
\min\ & \rlap{$\mathbb{E}\Big[\textstyle\sum_{n=0}^\infty \gamma^n \Big(c u_n - p\cdot(\xi_n+u_n-d_n)^- +$}&\hspace{1.5cm} \\
&\rlap{$\hspace{1.5cm}q(t_n)\cdot(\xi_n+u_n-d_n)^+ + c_f f_n x_n\Big)\Big]$}&\\
\text{s.t. }& x_{n+1} = (\xi_n + u_n - d_n)^+, & \text{for } n \geq 0&\\
&t_{n+1}=(\tau_n + 1), & \text{for } n \geq 0&\\
&\xi_n = (1-f_n)\cdot x_n, & \text{for } n \geq 0&\\
&\tau_n = (1-f_n)\cdot t_n, & \text{for } n \geq 0&\\
&t_n\leq k, & \text{for } n \geq 0&\\
& f_n\in\{0,1\}, u_n \in\mathbb{R}_+, & \text{for } n \geq 0&\hspace{1cm}
\end{aligned}
\label{eq:infinite_horz}
\end{equation}
The interpretation is that $p\cdot(\xi_n+u_n-d_n)^-$ is a penalty for not having stored enough water to meet demand, $q(t_n)\cdot(\xi_n+u_n-d_n)^+$ measures the health cost due to storing excess water, $c$ is the cost of the water, $c_f$ is the waste penalty for flushing too much water, $\gamma \in [0,1)$ is a discount factor, and the constraint $\textstyle\sum_{j=1}^k f_{n+j} \geq 1$ for $n\geq 0$ ensures that the storage tank is fully emptied at least once every $k$ days to maintain safe water quality. Since the function $q$ is strictly increasing, the $q(t_n)$ term indicates that the quality of water deteriorates as time passes between emptying the tank.

\subsection{Water Demand Model}

A commonly used stochastic distribution for water demand $d_n$ is the Poisson Rectangular Pulse (PRP) model \cite{rodriguez1987rectangular}, and the simplest version of this model is that the number of usage events has a Poisson distribution with rate $\lambda$ while the duration of water usage is the sum of exponential distributions with rate $\mu$.  The probability density function (PDF) of $d_n$ is given by
\begin{equation}
f(d) = \sum_{k=1}^\infty\frac{(\lambda\cdot\mu)^kd^{k-1}e^{(-\lambda-\mu d)}}{k!(k-1)!}\mathbf{1}[d>0] + e^{-\lambda}\delta(d)
\end{equation}
where $\lambda > 0$ is the usage event rate, $\mu > 0$ is the usage duration rate, $\mathbf{1}[\cdot]$ is the indicator function, and $\delta(\cdot)$ is the Dirac delta function. An alternative expression \cite{rodriguez1987rectangular} for the PDF of this distribution is 
\begin{equation}
f(d) = e^{(-\lambda -\mu d)}I_1(2\sqrt{\lambda\mu d})\sqrt{\frac{\lambda\mu}{d}}\mathbf{1}[d>0] + e^{-\lambda}\delta(d)
\end{equation}
where $I_1$ is the modified Bessel function of the first type.  Though this is an unusual distribution, generating random numbers is relatively straightforward: 

\begin{proposition}
Algorithm \ref{alg:rng} generates random numbers whose distribution is that of the PRP model.
\end{proposition}

\begin{proof}
This is true by construction of the algorithm and by the definition of the PRP model.
\end{proof}

\begin{algorithm}[t]
\begin{algorithmic}[1] 
\caption{Random Number Generator for PRP Model}\label{alg:rng}
\State initialize $d\leftarrow 0$
\State set $e \leftarrow $ from Poisson distribution with rate $\lambda$
\For {$\nu = 1,\ldots,e$}
\State set $\sigma\leftarrow$ from exponential distribution with rate $\mu$
\State set $d\leftarrow d + \sigma$
\EndFor
\end{algorithmic}
\end{algorithm}

\section{Reset Control Model and BiDS Algorithm}

\label{sec:rcmba}
An important characteristic of the local water storage control problem is that choosing $f_n =1$ in effect causes the state of the system to be reset to $(x_n,t_n) = (0,0)$.  In this section, we generalize the local water storage control problem into a broader class of stochastic optimal control problems with controlled resets and with constraints on the maximum time between rests.  Then we develop a new algorithm that uses binary search to compute the value function and hence solve this particular stochastic optimal control problem.

\subsection{Stochastic Optimal Control with Controlled Resets}

Let the subscript $n\in\mathbb{Z}$ denote time, and consider the discrete-time dynamical system 
\begin{equation}
\label{eqn:dyngen}
\begin{aligned}
x_{n+1} &= h(\xi_n,\tau_n,u_n,w_n)\\
t_{n+1} &= \tau_n +1\\
\xi_n & = x_n\cdot(1-r_n) + \zeta\cdot r_n\\
\tau_n & = t_n\cdot(1-r_n)
\end{aligned}
\end{equation}
where $x_n\times t_n \in \mathbb{R}^p\times\mathbb{Z}_+$ are states, $\xi_n\times \tau_n \in \mathbb{R}^p\times\mathbb{Z}_+$ are psuedo-states, $u_n\times r_n \in \mathbb{R}^q\times\{0,1\}$ are controls, $w_n \in \mathbb{R}^m$ are i.i.d. random variables, and $h : \mathbb{R}^p\times\mathbb{Z}_+\times\mathbb{R}^q\times\mathbb{R}^m\rightarrow\mathbb{R}$ is a deterministic function.  The interpretation is that the control $r_n = 1$ resets the system to a \emph{known} initial state $\zeta\in\mathbb{R}^p$, the state $t_n$ keeps track of how many time steps have passed since the last system reset, and $h$ describes the dynamics when there is no reset.

Let $\gamma\in[0,1)$ be a discount factor, and suppose that our goal is to solve the stochastic control problem
\begin{equation}
\label{eqn:genscp}
\begin{aligned}
\min\ & \rlap{$\mathbb{E}\Big[\textstyle \sum_{n=0}^\infty \gamma^n \Big(g(\xi_n,\tau_n,u_n,w_n) + s(x_n,t_n,w_n)\cdot r_n\Big)\Big]$}&\\
\text{s.t. }& (\ref{eqn:dyngen}),\, t_n \leq k,\, u_n\in\mathcal{U}, &\text{for } n \geq0\hspace{6cm}&\\
\end{aligned}
\end{equation}
where $g : \mathbb{R}^p\times\mathbb{Z}_+\times\mathbb{R}^q\times\mathbb{R}^m\rightarrow\mathbb{R}_+$ is a nonnegative and continuous stage cost, $s : \mathbb{R}^p\times\mathbb{Z}_+\times\mathbb{R}^m\rightarrow\mathbb{R}_+$ is a nonnegative and continuous reset cost with $s(\zeta,t,w)\equiv0$, the constraint $u_n \in \mathcal{U}$ restricts the possible control actions, and the constraint $t_n \leq k$ requires the system be reset at least once every $k$ time steps. This stochastic optimal control problem has special structure that can be used to develop a new approach for its solution.

We first characterize the optimal value function for (\ref{eqn:genscp}).  Let $J(x,t)$ be the optimal cost to go, and define $J_0 = J(\zeta,0)$.  Recall that $J(x,t)$ is defined as the minimum value of (\ref{eqn:genscp}) for the initial conditions $x_0 = x$ and $t_0 = t$.

\begin{proposition}
\label{prop:dpgenscp}
The dynamic programming equations for (\ref{eqn:genscp}) are given by
\begin{equation}
\label{eqn:ansbell}
\begin{aligned}
J(\zeta,0) &= \min_{u\in\mathcal{U}} \mathbb{E}\Big[ g(\zeta,0,u,w) + \gamma J(h(\zeta,0,u,w),1)\rlap{$\Big]$}\\
J(x,t) &= \min\Big\{J_0 + \mathbb{E}\big(s(x,t,w)\big),\\
&\hspace{0.5cm}\min_{u\in\mathcal{U}} \mathbb{E}\Big[ g(x,t,u,w) +\gamma J(h(x,t,u,w),t+1)\Big]\Big\}\\
J(x,k) &= J_0 + \mathbb{E}\big(s(x,k,w)\big)
\end{aligned}
\end{equation}
where the middle $J(x,t)$ holds for all $x$ and $t=0,\ldots,k-1$.
\end{proposition}

\begin{proof}
We prove this by induction on $t$.  Recalling that $s(\zeta,0,w) = 0$, we rewrite \eqref{eqn:genscp} for $(x_0,t_0) = (\zeta,0)$ as
\begin{equation}
\label{eqn:z0}
\begin{aligned}
\min\ &\rlap{$\mathbb{E}\Big[\textstyle g(\zeta,0,u_0,w_0) + \sum_{n=1}^\infty\gamma^n \Big(g(\xi_n,\tau_n,u_n,w_n) +$}&\\
&\rlap{$\qquad\qquad\qquad\qquad\qquad\qquad\qquad s(x_n,t_n,w_n)\cdot r_n\Big)\Big]$}&\\
\text{s.t. }& (\ref{eqn:dyngen}),\, t_n \leq k,\, u_n\in\mathcal{U}, &\text{for } n \geq0\hspace{6cm}&\\
\end{aligned}
\end{equation}
But since $(x_0,t_0) = (\zeta,0)$, we have $\xi_0 = \zeta$ and $t_1 = 1$ for both $r_0 = 0$ and $r_0 = 1$.  Thus the above simplifies to the expression for $J(\zeta,0)$ given in (\ref{eqn:ansbell}).

Next consider the case with $(x_0,t_0) = (x,k)$.  We must have $r_0 =1$ because otherwise $t_1 = k+1$ which violates the constraint $t_n \leq k$.  So $\xi_0 = \zeta$, and we can rewrite \eqref{eqn:genscp} as
\begin{equation}
\begin{aligned}
\min\ & \rlap{$\mathbb{E}\Big[\textstyle s(x,k,w_0) + g(\zeta,0,u_0,w_0) +$}\\
& \rlap{$\quad\textstyle \sum_{n=1}^\infty \gamma^n \Big(g(\xi_n,\tau_n,u_n,w_n)  +  s(x_n,t_n,w_n)\cdot r_n\Big)\Big]$}\\
\text{s.t. }& (\ref{eqn:dyngen}),\, t_n \leq k,\, u_n\in\mathcal{U}, &\text{for } n \geq0\hspace{6cm}&\\
\end{aligned} 
\end{equation}
Clearly we have $J(x,k) = J(\zeta,0) + \mathbb{E}[s(x,k,w)]$, since the reset cost $s(x,k,w)$ does not depend on the control $u_0$.

 Now consider the case with $(x_0,t_0) = (x,k-1)$, and let $J_0(x,k-1)$ and $J_1(x,k-1)$ denote the optimal cost to go under $r=0$ and $r=1$, respectively.  By Bellman's principle of optimality we have $J(x,k-1) = \min\{J_0(x,k-1),J_1(x,k-1)\}$. For $J_1(x,k-1)$, we have $\xi_0 = \zeta$ by definition, and so a similar argument to the one for $J(x,k)$ gives that $J_1(x,k-1) = J_0 + \mathbb{E}[s(x,k-1,w)]$. Next, observe that we can express $J_0(x,k-1)$ as
\begin{equation}
\begin{aligned}
\min\ & \rlap{$\mathbb{E}\Big[\textstyle g(x,k-1,u_0,w_0)  +$}\\& \rlap{$\textstyle\sum_{n=1}^\infty \gamma^n \Big(g(\xi_n,\tau_n,u_n,w_n)  +  s(x_n,t_n,w_n)\cdot r_n\Big)\Big]$}\\
\text{s.t. }& (\ref{eqn:dyngen}),\, t_n \leq k,\, u_n\in\mathcal{U}, &\text{for } n \geq0\hspace{6cm}&\\
\end{aligned} 
\end{equation}
Simplifying gives $J_0(x,k-1) = \min \mathbb{E}[\textstyle g(x,0,u_0,w_0) + \gamma J(h(x,k-1,u_0,w_0),k)]$. Hence $J(x,k-1)$ satisfies the middle equation for $J(x,t)$ in (\ref{eqn:ansbell}). Now suppose that the middle equation for $J(x,t)$ in (\ref{eqn:ansbell}) is satisfied up to $0 < t+1 \leq k-1$: We will show this implies that the middle equation for $J(x,t)$ in (\ref{eqn:ansbell}) is satisfied for $t$.  Indeed, using the same argument as for the base case gives this desired step.  The final result follows by applying induction.
\end{proof}

\subsection{Design of the BiDS Algorithm}

To solve these dynamic programming equations and compute the optimal policy, we develop our Binary Dynamic Search (BiDS) algorithm that is presented in Algorithm \ref{alg:bids_alg}. The key feature of BiDS is that unlike value iteration \cite{bertsekas1995}, which requires finding a fixed point over the space of functions by repeated application of a contraction mapping, BiDS performs a binary search in order to find a fixed (vector-valued) point of an appropriately constructed mapping.

\begin{algorithm}[t]
\begin{algorithmic}[1] 
\caption{Binary Dynamic Search (BiDS) Algorithm}\label{alg:bids_alg}
\State initialize $ \underline{v} \leftarrow 0$ and $\overline{v} \leftarrow (\frac{1}{1-\gamma})\min_u \mathbb{E}[g(\zeta,0,u,w_0) + 
\gamma\cdot s(h(\zeta,0,u,w_0),1,w_1)]$
\Repeat
\State set $v \leftarrow (\overline{v}+\underline{v})/2$
\State set $V(x,k,v) = v + \mathbb{E}[s(x,w,k)]$
\For {$t = (k-1),(k-2),\ldots,0$}
\State set $V(x,t,v) = \min\big\{v + \mathbb{E}[s(x,w,t)],\qquad\qquad\qquad\linebreak\hspace*{4.4cm}\min_{u\in\mathcal{U}} \mathbb{E}[g(x,t,u,w) +\linebreak\hspace*{3cm}\gamma V(h(x,t,u,w),t+1,v)]\big\}$
\EndFor
\State set $\Upsilon(v) = \min_{u\in\mathcal{U}} \mathbb{E}[g(\zeta,0,u,w) +\qquad\qquad\qquad\qquad\linebreak\hspace*{4.5cm}\gamma V(h(\zeta,0,u,w),1,v)]$
\If{$v < \Upsilon(v)$}
\State set $\overline{v}\leftarrow v$
\Else
\State set $\underline{v}\leftarrow v$
\EndIf
\Until {$(\overline{v}-\underline{v}) \leq \epsilon$}
\State set $v^*=(\overline{v}+\underline{v})/2$
\end{algorithmic}
\end{algorithm}
Our main result about the BiDS algorithm concerns its finite termination and convergence properties. To show this, we first study the recursion calculations within BiDS.
\begin{proposition}
\label{prop:approxval}
Consider the recurrence relation given by
\begin{equation} \label{eq:bids_recur}
\begin{aligned}
\Upsilon(v) &= \min_{u\in\mathcal{U}} \mathbb{E}\Big[g(\zeta,0,u,w) + \gamma V(h(\zeta,0,u,w),1,v)  \Big] \\
V(x,t,v) &= \min\Big\{ v + \mathbb{E}\big(s(x,t,w)\big), \\
&\hspace{-0.33cm}\min_{u\in\mathcal{U}} \mathbb{E} \Big[g(x,t,u,w) + \gamma V(h(x,t,u,w),t+1,v) \Big]  \Big\} \\
V(x,k,v) &= v + \mathbb{E}\big(s(x,k,w)\big)
\end{aligned}
\end{equation}
where the middle equation for $V(x,t,v)$ holds for all $x$ and $t=0,\ldots,k-1$.  Then $\Upsilon(v)$ is monotone increasing in $v$, we have that $V(x,t,J_0) = J(x,t)$, and $v= J_0$ is the unique (vector-valued) fixed point of $\Upsilon(v)$. 
\end{proposition}

\begin{proof}
To prove the first result, note that $V(x,k,v)$ is monotone increasing by construction.  Next, suppose $V(x,t+1,v)$ is monotone increasing in $v$.  The first term in the $\min$ defining $V(x,t,v)$ is monotone increasing by construction, and the second term in the $\min$ defining $V(x,t,v)$ is monotone increasing by assumption.  Thus $V(x,t,v)$ is monotone increasing in $v$, and this result extends to all $t = 0,\ldots,k-1$ by induction.  Lastly, observe that $\Upsilon(v)$ is monotone increasing in $v$ since its first term is independent of $v$ and since the second term $V(h(\zeta,0,u,w),1,v)$ was shown to be monotone increasing in $v$.

To prove the second result, we assume $v= J_0$ and use induction on $t$.  First note $V(x,k,J_0) = J_0 + \mathbb{E}(s(x,k,w)) = J(x,k)$ by Proposition \ref{prop:dpgenscp}.  Next suppose that $V(x,t+1,J_0) = J(x,t+1)$, and note that by definition 
\begin{align}
&V(x,t,J_0) = \min\Big\{ J_0 + \mathbb{E}\big(s(x,t,w)\big), \\
&\qquad\min_{u\in\mathcal{U}} \mathbb{E} \Big[g(x,t,u,w) + \gamma V(h(x,t,u,w),t+1,J_0) \Big]  \Big\}.\nonumber
\end{align}
Comparing to Proposition \ref{prop:dpgenscp} shows $V(x,t,J_0)=J(x,t)$, and so the result for $V(x,t,J_0)$ for all $x$ and $t=0,\ldots,k-1$ follows by induction.  Lastly, inserting the relation $V(x,1,J_0)=J(x,1)$ into the definition for $\Upsilon(J_0)$ gives $\Upsilon(J_0) = J_0$ by comparison with Proposition \ref{prop:dpgenscp}.

To show the third result, let $\nu$ be any value such that $\nu=\Upsilon(\nu)$. (Note that we have shown above that such a $\nu$ must exist.)  Then we can rewrite \eqref{eq:bids_recur} as
\begin{equation}
\begin{aligned}
\nu =\Upsilon(\nu)&= \min_{u\in\mathcal{U}} \mathbb{E}\Big[g(\zeta,0,u,w) + \gamma V(h(\zeta,0,u,w),1,\nu)  \Big] \\
V(x,t,\nu) &= \min\Big\{ \Upsilon(\nu) + \mathbb{E}\big(s(x,t,w)\big), \\
&\hspace{-0.33cm}\min_{u\in\mathcal{U}}\mathbb{E} \Big[g(x,t,u,w) + \gamma V(h(x,t,u,w),t+1,\nu) \Big]  \Big\} \\
V(x,k,\nu) &= \Upsilon(\nu) + \mathbb{E}\big(s(x,k,w)\big)
\end{aligned}
\end{equation}
Comparing these to Proposition \ref{prop:dpgenscp} shows that $V(x,t,\nu)$ is a fixed point of the dynamic programming equations for (\ref{eqn:genscp}).  Thus $V(x,t,\nu) = J(x,t)$, and $J(x,t)$ is unique by Proposition 4.1.5 of \cite{bertsekas1995}.  As a result, we have  $\nu = \Upsilon(\nu)=V(\zeta,0,\nu) = J_0$.  This means that $\nu$ is uniquely defined.
\end{proof}

Before we prove our main result about the BiDS algorithm, we need an additional lemma.

\begin{lemma}
\label{lemma:minlem}
Suppose $|a'-a| \leq \epsilon$ and $|b'-b| \leq \epsilon$.  Then we have that $|\min\{a',b'\}-\min\{a,b\}|\leq \epsilon$.
\end{lemma}

\begin{proof}
Let $c' = \min\{a',b'\}$ and $c = \min\{a,b\}$.  We have four cases:  (1) $c' = a'$ and $c = a$, then $|c'-c|\leq\epsilon$; (2) $c' = b'$ and $c = b$, then $|c'-c|\leq\epsilon$; (3) $c' = a'$ and $c = b$, then $c'-c=a'-b\geq a'-a\geq-\epsilon$ and $c'-c=a'-b\leq b'-b\leq\epsilon$; and (4) $c' = b'$ and $c = a$, then $c'-c=b'-a\geq b'-b\geq-\epsilon$ and $c'-c=b'-a\leq a'-a\leq\epsilon$.
\end{proof}

We can now show our main result about BiDS:

\begin{theorem}
\label{thm:bids}
The BiDS algorithm terminates in finite time when $\epsilon > 0$, and the solution $v^*$ returned by BiDS is such that $\max_{x}\big|V(x,t,v^*) - J(x,t)\big| \leq \epsilon$ for all $t = 0,\ldots,k$.
\end{theorem}
\begin{proof}
Since the state costs and reset costs are nonnegative, Proposition 4.1.5 from \cite{bertsekas1995} implies that the value function $J(x,t)$ is the unique fixed point of the dynamic programming equations (\ref{eqn:ansbell}) for (\ref{eqn:genscp}).  Our first intermediate claim is that $J_0$ is finite.  In particular, define
\begin{multline}
\overline{J} = \textstyle\Big(\frac{1}{1-\gamma}\Big)\min_{u\in\mathcal{U}} \mathbb{E}\Big[g(\zeta,0,u,w_0) +\\ \gamma\cdot s(h(\zeta,0,u,w_0),1,w_1)\Big],
\end{multline}
and observe that $\overline{J}$ is finite since the objective is nonnegative and the optimization problem is always feasible.  Next note $J_0\geq 0$ since the stage costs and reset costs are nonnegative.  Furthermore, using the dynamic programming equations (\ref{eqn:ansbell}) from Proposition \ref{prop:dpgenscp} gives that $J(x,1) \leq J_0 + \mathbb{E}(s(x,1,w))$.  Again using the dynamic programming equations (\ref{eqn:ansbell}) gives
\begin{multline}
J_0 \leq \min_u \mathbb{E}\Big[g(\zeta,0,u,w_0) +\\ \gamma(J_0 + s(h(\zeta,0,u,w_0),1,w_1)\Big].
\end{multline}
Solving for $J_0$ gives $J_0 \leq \overline{J}$.  Since we showed $J_0\in[0,\overline{J}]$, this implies $J_0$ is finite.  Next observe that the BiDS algorithm is a binary search over the range $[0,\overline{J}]$.  Recalling that $V(\zeta,0,v)$ is monotone increasing in $v$ by Proposition \ref{prop:approxval}, we must have that the binary search terminates in a finite number of steps when $\epsilon > 0$.  This proves the first part of the theorem.

We use induction to prove the second part of the theorem.  First note that by Propositions \ref{prop:dpgenscp} and \ref{prop:approxval} we have
\begin{equation}
\textstyle\max_x\big|V(x,k,v^*) - J(x,k)\big| = \big|v^* - J_0\big| \leq \epsilon
\end{equation}
where the inequality follows by the termination condition of BiDS.  Now suppose we have that $\max_x|V(x,t+1,v^*) - J(x,t+1)| \leq \epsilon$.  Then a similar calculation gives $\max_x\big|v^* + \mathbb{E}(s(x,t,w)) - J_0 - \mathbb{E}\big(s(x,t,w)\big)\big| \leq \epsilon$ and that $\max_{x,u}\big|g(x,t,u,w) + \gamma V(h(x,t,u,w),t+1,v^*) - g(x,t,u,w) - \gamma J(h(x,t,u,w),t+1)\big| \leq \epsilon$.  And so Lemma \ref{lemma:minlem} implies that $\max_x|V(x,t,v^*)- J(x,t)| \leq \epsilon$ since expectation preserves distances.  Thus $\max_x|V(x,t,v^*)- J(x,t)| \leq \epsilon$ for all $t=0,\ldots,k$ by induction. 
\end{proof}

We can use the BiDS algorithm to compute the optimal policy by taking the minimizing arguments of $V(x,t,v^*)$.

\begin{corollary}
Let $v^*$ be the value returned by BiDS for a given $\epsilon > 0$.  If $\textstyle(u,r)^*(x,t,\epsilon)\in\argmin V(x,t,v^*)$ for $(x,t)\neq(\zeta,0)$ and $\textstyle(u,r)^*(\zeta,0,\epsilon)\in\argmin \Upsilon(v^*)$, then $\textstyle\limsup_{\epsilon\downarrow 0} (u,r)^*(x,t,\epsilon) \subseteq \argmin J(x,t)$ for $t=0,\ldots,k$.
\end{corollary}

\begin{proof}
Theorem \ref{thm:bids} says $V(x,t,v^*)$ converges uniformly to $J(x,t)$, and that $\Upsilon(v^*)$ converges uniformly to $J(\zeta,0)$.  Next recall Fatou's lemma \cite{lieb2001analysis} implies lower semicontinuity (lsc) of nonnegative functions is preserved under expectation.  This means $V(x,k,v^*)$ is lsc in $x$.  Now suppose $V(x,t+1,v^*)$ is lsc in $x$.  Then $v^* + \mathbb{E}(s(x,t,w))$ and $\mathbb{E}[g(x,t,u,w) + \gamma V(h(x,t,u,w),t+1,v)]$ are lsc in $(x,u)$ by Fatou's lemma.  Since lsc is preserved under minimization \cite{rockafellar2009}, this means $V(x,t,v^*)$ is lsc in $x$.  By induction, $V(x,t,v^*)$ for $t=0,\ldots,k$ is lsc in $x$.  Hence, Proposition 7.15 of \cite{rockafellar2009} says the objectives of $V(x,t,v^*)$ and $\Upsilon(v^*)$ epi-converge.  So the result follows by Theorem 7.33 of \cite{rockafellar2009}.
\end{proof}

In words, the above result says that as $\epsilon\rightarrow 0$ then BiDS generates a sequence of $v^*$ such that the minimizers of $V(x,t,v^*)$ and $\Upsilon(v^*)$ converge to minimizers of $J(x,t)$.


\section{Numerical Solution of Water Storage Model}

\label{sect:nswsm}

In this section, we use BiDS to numerically solve the water storage model (\ref{eq:infinite_horz}) from Sect. \ref{sec:wat_stor_mode}. Computations were performed in MATLAB 2016b on a laptop computer with a 2.4GHz processor and 16GB of RAM. We generated random demand from the PRP model using Algorithm \ref{alg:rng} with Poisson rate $\lambda = 40$ (i.e., average of 40 water usage events for a family) and exponential rate $\mu = 2/\text{L}$ (i.e., average water usage event consumes $500$mL of water). The costs in Indian Rupees were 0.25/L, 0.5/L, 0.5/L for the purchasing, shortage, and flushing costs of water, respectively. The holding cost function in Indian Rupees was $q(t_n) = 15t_n$/L, which satisfies the monotonicity assumption and is high compared to the other costs to represent a sharp deterioration in the quality of water.
\begin{figure*}[t]
\includegraphics[trim={0.75in 0.5in 0.8in 0in},clip,scale=0.9]{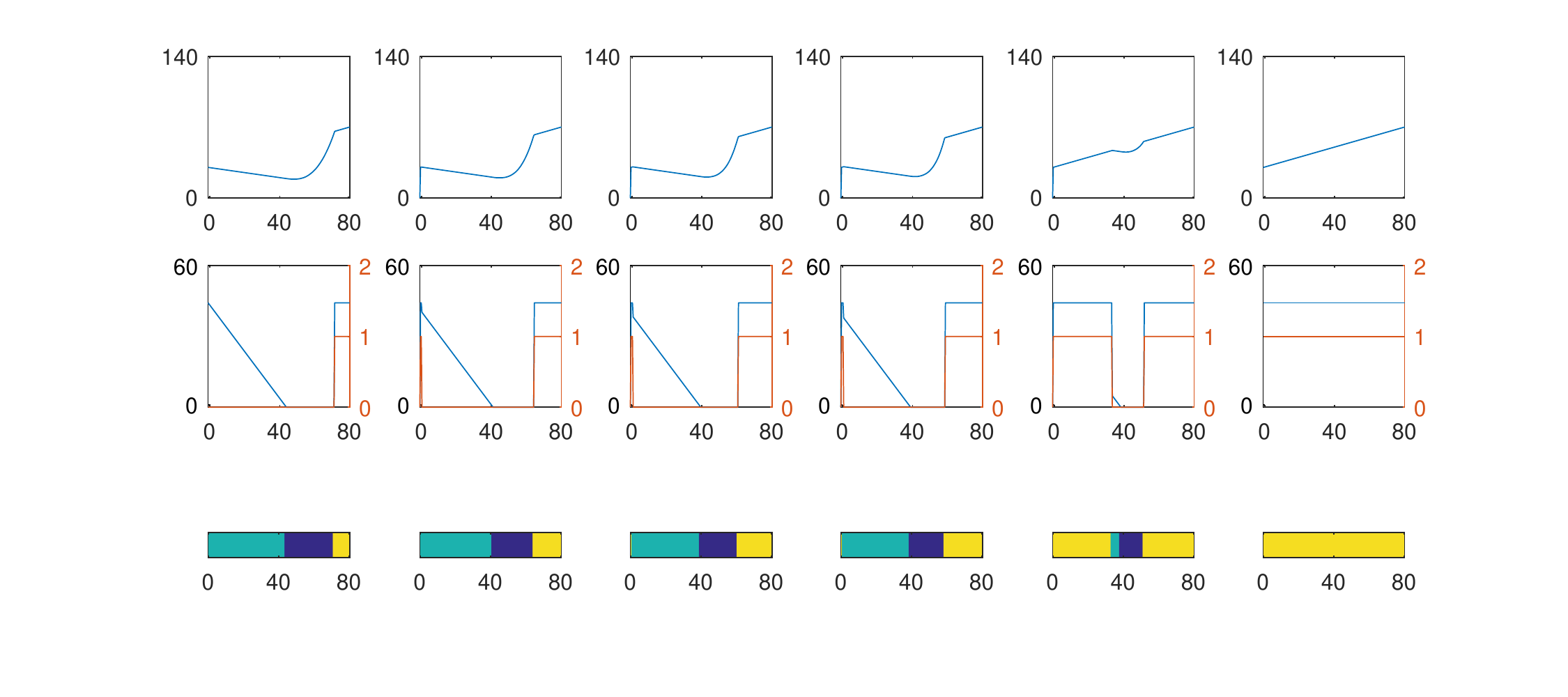}
\caption{The top row shows the value function $V(x,t,v^*)$ with $\epsilon=10^{-1}$, the second row shows $u^*(x,t)$ in blue and $f^*(x,t)$ in red, and the bottom row shows the different zones for which the optimal policy corresponds with cyan corresponding to ordering water and not flushing, dark blue to not doing anything, and yellow to flushing the tank and reordering water. For all plots the state $x$ (or the amount of water in the tank) is on the $x$-axis, and the subplots from left to right correspond to states $t =1,2,3,4,5,6$, respectively}
\label{fig:plots}
\end{figure*}

The BiDS algorithm was used to solve the local water storage problem with a maximum of $k = 6$ days between flushing and with a discount factor of $\gamma = 0.8$, and the results are shown in Fig. \ref{fig:plots}. These results show the optimal policy has a clear structure, and this control policy is easily implementable via a lookup table that could be distributed through paper pamphlets or internet websites.  At each $t$, there are at most three thresholds that separate the control policy into four regions. The first region is where the amount of water in the tank is so small that the flushing cost is negligible, and hence it is optimal to flush the tank and order a one step optimal quantity. Next is a region with enough water in the tank were it is optimal to order up to some quantity and still consume most of the water before it becomes unsafe to drink. Then is a do-nothing region where it is not optimal to flush the tank or order any more water since there is enough water in the tank to satisfy possible future consumption. Finally, there is a region where there is so much water in the tank that it is almost certain to never be consumed, and thus it is optimal to flush this water from the tank and reorder up to some optimal quantity.

\section{Conclusion}
In this paper, we studied the problem of managing local water storage in the developing world. We first presented a stochastic optimal control formulation of this problem, and gave an algorithm for sampling from a PRP model that describes daily water demand.  We showed this water storage problem can be generalized to stochastic optimal control problems with controlled resets to a single state and with constraints on the maximum time between resets. To solve this class of problems, we developed the BiDS algorithm that performs binary search to obtain the optimal value function and corresponding optimal policy. We showed this algorithm converges in finite time and is guaranteed to find the optimal value with $\epsilon$ precision. Finally, we concluded by applying the BiDS algorithm to numerical solve the problem of managing local water storage. Our results suggest that the optimal control policy for this problem has a well defined structure, which we leave to be studied in future work.
%


\bibliographystyle{IEEEtran}
\bibliography{IEEEabrv,waterstore}

\end{document}